\documentclass[12pt,reqno]{amsart}

\usepackage{comment}
\usepackage{amsfonts}
\usepackage{amssymb}
\usepackage{graphicx}
\usepackage{amsmath,mathtools}
\usepackage{latexsym}
\usepackage{amscd}
\usepackage{xypic}
\usepackage{mathrsfs}
\usepackage{enumitem} 
\usepackage{braket}
\usepackage[margin=1.4in]{geometry}
\usepackage{esint}
\usepackage{amsthm}
\usepackage{accents}
\usepackage[protrusion=true,expansion,stretch=5]{microtype}
\usepackage{tikz}
\usetikzlibrary{snakes}
\usepackage{parskip}
\usepackage{stmaryrd}

\usepackage[pdftex]{hyperref}

\setlist{itemsep=2pt}

\allowdisplaybreaks

\newtheorem{prop}{Proposition}
\newtheorem{theo}[prop]{Theorem}

\theoremstyle{definition}

\newtheorem{rema}[prop]{Remark}

\numberwithin{prop}{section}
\numberwithin{equation}{section}

\newcommand{\CC}{\mathbf{C}}

\newcommand{\RR}{\mathbf{R}}
\renewcommand{\SS}{\mathbf{S}}

\newcommand{\YY}{\mathbf{Y}}
\newcommand{\ZZ}{\mathbf{Z}}

\DeclareMathOperator{\length}{length}

\DeclareMathOperator{\dist}{dist}

\newcommand{\red}[1]{\textcolor{red}{#1}}

\newcommand{\source}{\textcolor{red}{source}}
\newcommand{\acosh}{\text{arccosh}}
\newcommand{\nl}{\newline}

\DeclareMathOperator{\sys}{sys}

\allowdisplaybreaks

\begin{document}

\title[Almgren's Three-Legged Starfish]{Almgren's Three-Legged Starfish}

\author{Christos Mantoulidis}
\address{Department of Mathematics, Rice University, Houston, TX 77005}
\email{christos.mantoulidis@rice.edu}

\author{Jared Marx-Kuo}
\address{Department of Mathematics, Rice University, Houston, TX 77005}
\email{jm307@rice.edu}

\begin{abstract}
    In this note we use classical tools from min-max and hyperbolic geometry to substantiate a folklore example in Almgren--Pitts min-max theory, the three-legged starfish metric on a 2-sphere, whose systolic length, Almgren--Pitts width, and Gromov--Guth width are attained by ``figure-eight'' geodesics. We also recover a hyperbolic geometry fact about ``figure-eight'' geodesics using min-max.
\end{abstract}

\maketitle

\section{Introduction}

In 1965, Almgren \cite[\S 15-8]{almgren1965theory} (cf. \cite[p. 20-21]{pitts1981existence}) proposed the existence of a Riemannian metric on $\SS^2$ such that (what is now referred to as) the ``Almgren--Pitts min-max width'' of the metric was realized by a ``figure-eight'' (``$8$'') geodesic. The metric was dubbed the ``three-legged starfish'' due to its resemblance.  While easy to visualize, an explicit metric was never constructed. Nonetheless, the expectation that a width can be realized by a figure-eight geodesic is an important feature of Almgren--Pitts min-max theory. The need to understand and control these configurations  led Pitts to his contribution of the notion of ``almost minimizing (in annuli)'' cycles (\cite[\S 1.2]{pitts1981existence}) that are seminal to the regularity aspect of Almgren--Pitts theory. 



In this note, we produce such three-legged starfish using hyperbolic geometry. 

\begin{theo} \label{theo:intro.starfish}
    The exist Riemannian metrics $g$ on $\SS^2$ obtained from the complete (finite-area) hyperbolic metric on a thrice punctured $\SS^2$ after truncating its cusps and capping with convex caps, with this property: their shortest closed geodesics are all non-simple, and are in fact isometric images of one of the three figure-eight geodesics in the hyperbolic surface. 
\end{theo}

Hyperbolic geometry guarantees the existence of a unique complete hyperbolic (curvature $-1$) metric on a thrice punctured $\SS^2$. The punctures turn into to cusps, and figure-eight geodesics go around pairs of cusps and with length $4 \operatorname{arcsinh} 1 = 2 \operatorname{arccosh} 3$. It is well-known that they have least length among \textit{all} closed geodesics in the surface. (See Theorem \ref{theo:starfish.systole}.) Our min-max argument, in the compactified setting necessary for Theorem \ref{theo:intro.starfish}, provides an interesting alternative approach to this  hyperbolic geometry fact, with quite minimal hyperbolic geometry input (see Proposition \ref{prop:hyperbolic}).

Now let us recall some notation:
\begin{itemize}
    \item The systolic length\footnote{In spaces other than $\SS^2$, one might require the geodesics to be homotopically nontrivial.} $\sys(\SS^2, g)$ is the least length among all closed geodesics. At least one closed geodesic exists by classical min-max work of Birkhoff \cite{birkhoff1927dynamical} using 1-parameter sweepouts in the space of closed curves in $\SS^2$.
    \item The Almgren--Pitts width $\omega_{\textnormal{AP}}(\SS^2, g)$ is the least length among all mod-2 1-cycles produced by Almgren--Pitts min-max with homotopically nontrivial loops in the space of mod-2 1-cycles in $\SS^2$.
    \item The Gromov--Guth width $\omega_{\textnormal{GG}}(\SS^2, g)$ is the least length among all mod-2 1-cycles produced by Almgren--Pitts min-max with sweepouts that cohomologically detect the generator of the space of mod-2 1-cycles in $\SS^2$. This is also known as the ``$1$-width'' in relation to the Gromov--Guth ``$p$-widths'' that were further popularized by Marques--Neves's work on minimal hypersurfaces (cf. \cite{gromov1983filling, gromov2002isoperimetry, gromov2006dimension, guth2009steenrod, marquesneves2020applications}).
\end{itemize}

The work of Calabi--Cao \cite{calabi1992simple} and an excursion into double-well phase transitions via Gaspar--Guaraco \cite{gaspar2018allen}, Mantoulidis \cite{mantoulidis2021allen} or Chodosh--Mantoulidis \cite{chodosh2023p}, and Dey \cite{dey2022apac} imply:

\begin{theo} 
\label{theo:intro.width}
    Let $g$ be a Riemannian metric on $\SS^2$. Then:
    \begin{equation} \label{eq:intro.width.leq}
        \sys(\SS^2, g) \leq \omega_{\textnormal{GG}}(\SS^2, g) \leq  \omega_{\textnormal{AP}}(\SS^2, g).
    \end{equation}
    If $\sys(\SS^2, g)$ is attained by at least one nonsimple closed geodesic, then
    \begin{equation} \label{eq:intro.width.eq}
    \sys(\SS^2, g) = \omega_{\textnormal{GG}}(\SS^2, g) = \omega_{\textnormal{AP}}(\SS^2, g).
    \end{equation}
\end{theo}

We note: 
\begin{enumerate}
    \item[(1)] In the setting of nonnegatively curved metrics on $\SS^2$, Calabi--Cao \cite[Theorem D]{calabi1992simple} proved that only simple closed geodesics attain $\sys(\SS^2, g)$ \textit{and} that $\sys(\SS^2, g) = \omega_{\textnormal{AP}}(\SS^2, g)$. Thus, \eqref{eq:intro.width.leq} still implies \eqref{eq:intro.width.eq}. However, \eqref{eq:intro.width.eq} is \textit{not} otherwise generally true for $(\SS^2, g)$ whose $\sys(\SS^2, g)$ is only attained by simple closed geodesics as, e.g., long rotationally symmetric dumbell metrics on $\SS^2$ (whose curvature changes sign) show.
    \item[(2)] In the setting of closed hyperbolic surfaces, Lima \cite{lima2025first} recently showed that figure-eights can attain $\omega_{\textnormal{GG}}(\Sigma, g)$ (suitably extended) for $\Sigma$ of any genus.  For certain surfaces, he proved that $\omega_{\textnormal{GG}}(\Sigma, g)$ equals the least length of any \textit{separating} closed geodesic since. (Of course, the shortest among \textit{all} closed geodesics in a closed hyperbolic surface is nonseparating.) Our work is independent of his but we were inspired by his results. 
\end{enumerate}  

We emphasize that Theorems \ref{theo:intro.starfish}, \ref{theo:intro.width} highlight the exceptional behavior of two-dimensional Almgren--Pitts min-max theory:

\begin{enumerate}
    \item[(3)] If, in dimension two, we use the mapping approach to min-max theory, as in Birkhoff's work, then Chambers--Liokumovich \cite{chambers2019optimal} proved that the min-max width of every $(\SS^2, g)$ is necessarily attained by a {simple} closed geodesic,  answering a question of Freedman.
    \item[(4)] In dimensions three through seven, the least area minimal hypersurface of any closed $(M^n, g)$ (at least one such hypersurface exists by \cite{pitts1981existence, schoensimon1981stable}) is necessarily embedded. This generalization of Calabi--Cao's two-dimensional result was proven by Song \cite{song2018embeddedness}. In positive Ricci curvature and in the same dimensions, Zhou \cite{zhou2015min} previously proved, even more so in the spirit of Calabi--Cao, that this least area will equal the Almgren--Pitts width in the absense of nonorientable hypersurfaces. 
\end{enumerate}

\textbf{Acknowledgments}. The authors thank Chris Leininger, Xin Zhou, Yevgeny Liokumovich, and Otis Chodosh for helpful discussions and references. The first author was supported by grant NSF DMS 2403728.

\section{A model}

Consider a Riemannian metric $g$ on $I \times \SS^1$, with $I \subset \RR$ an interval, of the form
\begin{equation} \label{eq:model.curv.metric}
    f(\rho)^2 \, d\rho^2 + e^{2\rho} d\theta^2,
\end{equation}
where $f : I \to (0, \infty)$ is smooth. For $f = 1$, this is a hyperbolic (curvature $-1$) metric in horocyclic coordinates, and for $f(\rho) = e^{\rho}$, it is a flat (curvature $0$) metric $dr^2 + r^2 \, d\theta^2$ after the change of variables $r = e^\rho$. An elementary computation (e.g., the first variation formula) shows that the 
the geodesic curvature of the circles $\{ \rho = \text{const} \}$ is given by
\begin{equation} \label{eq:model.curv.extrinsic}
    k = f(\rho)^{-1};
\end{equation}
that is, the geodesic curvature vector is $-k\nu = -f^{-1} f^{-1} \partial_\rho = -f^{-2} \partial_\rho$. 

\section{The hyperbolic starfish}

Let us denote by
\[ (\YY, h) = (\SS^2 \setminus \{ c_1, c_2, c_3 \}, h) \]
the $Y$-shaped complete (noncompact) hyperbolic surface obtained by endowing a thrice punctured sphere with a hyperbolic metric $h$ of curvature $-1$. This is our hyperbolic starfish and is unique up to isometries. 

The three distinguished points $c_i \in \SS^2$ produce three cusps $\CC_i \subset \YY$. We recall from elementary hyperbolic geometry that $\CC_i$ as above are pairwise disjoint and, in horocyclic coordinates about $c_i$:
\begin{equation} \label{eq:cusp.metric}
    \CC_i = (-\infty, \log 2] \times \SS^1, \; h|_{\CC_i} = d\rho^2 + e^{2\rho} \, d\theta^2.
\end{equation}
If, for $\sigma \in (-\infty, \log 2]$, we denote the ``$\sigma$-thin'' portions of $\CC_i$ by
\begin{equation} \label{eq:cusp.thin}
    \CC_i^{\leq \sigma} = \CC_i \cap \{ \rho \leq \sigma \},
\end{equation}
then it follows from \eqref{eq:cusp.metric} and \eqref{eq:model.curv.extrinsic} that 
\begin{equation} \label{eq:cusp.foliation}
    \sigma \mapsto \partial \CC_i^{\leq \sigma} \text{ are strictly $h$-convex curves foliating } \CC_i.
\end{equation}
To be clear, this means that the geodesic curvature vectors of these curves with respect to $h$ points strictly toward the noncompact end (``$c_i$'').

We record the following elementary facts from hyperbolic geometry. By slight abuse of notation we denote by $[c_1]$, $[c_2]$, $[c_3]$ representatives in $\pi_1(\YY)$ of three small consistently oriented simple closed curves around $c_1$, $c_2$, $c_3$ in $\SS^2$, then
\[ \pi_1(\YY) = \langle [c_1], [c_2], [c_3] \; | \; [c_1] [c_2] [c_3] = 1 \rangle. \]
The geodesic representatives of the free homotopy classes $[c_i][c_j]^{-1}$, $i \neq j$, have a single self-intersection and will be referred to as standard figure-eight geodesics.

\begin{prop} \label{prop:hyperbolic}
    Let $\gamma$ be a closed geodesic in $(\YY, h)$. Then:
    \begin{enumerate}
        \item $\gamma$ has at least one self-intersection.
        \item $\gamma$ has one self-intersection if and only if it is a standard figure-eight geodesic.
    \end{enumerate}
\end{prop}
\begin{proof}
    (1) This is \cite[Theorem 4.4.6]{buser2010spectra}.

    (2) 
    %
    Having a single self-intersection implies that $\gamma = \gamma_1 \gamma_2$ where $\gamma_1$, $\gamma_2$ are simple geodesic loops with only a vertex in common. Therefore, for $k = 1, 2$, we have the homotopy class equivalence $[\gamma_k] = [c_{i_k}]^{e_k}$ for $i_k \in \{ 1, 2, 3 \}$ and $e_k = \pm 1$. Among these, the only ones that are not homotopically trivial or homotopic to a power of a $[c_i]$ are precisely $[c_i][c_j]^{-1}$, $i \neq j$.
\end{proof}

%

The following result gives an elegant description of $\sys(\YY, h)$, the least length among closed geodesics in $(\YY, h)$. We avoid relying on it and, in fact, we show in the last section how to rederive it a posteriori from our min-max conclusions.

\begin{theo}[{\cite{yamada1982marden}}]  \label{theo:starfish.systole}
    The standard figure-eight geodesics attain $\sys(\YY, h)$. 
\end{theo}

\section{Truncated-and-capped hyperbolic starfish}

We change the hyperbolic metric $h$ inside $\CC_i$ and replace it with a new incomplete metric on $\CC_i$ whose metric completion extends to a smooth metric on $\SS^2$.

Fix $-\infty < \rho_* < \log 2$. Modify $h$ on $\cup_{i=1}^3 \CC_i$ to equal, instead of \eqref{eq:cusp.metric},
\begin{equation} \label{eq:cusp.trunc.metric}
    h_*|_{\CC_i} = f_*(\rho)^2 \, d\rho^2 + e^{2\rho} \, d\theta^2,
\end{equation}
where $f_* : (-\infty, \log 2] \to (0, \infty)$ is a smooth function satisfying:
\begin{equation} \label{eq:f.trunc.one}
    f_*(\rho) = 1 \text{ for all } \rho \in [\rho_*, \log 2],
\end{equation}
\begin{equation} \label{eq:f.trunc.zero}
    f_*(\rho) = e^\rho \text{ for all } \rho \in (-\infty, \rho_*-1].
\end{equation}

Recall, from \eqref{eq:cusp.thin}, the notation $\CC_i^{\leq \sigma}$ for the $\sigma$-thin portion of $\CC_i$. It follows from
\eqref{eq:cusp.trunc.metric}, \eqref{eq:f.trunc.one} 
that:
\begin{equation} \label{eq:cusp.trunc.isometric}
    h_* = h \text{ on } \YY \setminus \cup_{i=1}^3 \CC_i^{\leq \rho_*},
\end{equation}
Likewise, \eqref{eq:cusp.trunc.metric} and \eqref{eq:model.curv.extrinsic} imply, for $i=1$, $2$, $3$, that
\begin{equation} \label{eq:cusp.trunc.foliation}
     \sigma \mapsto \partial \CC_i^{\leq \sigma} \text{ are strictly $h_*$-convex curves foliating } \CC_i.   
\end{equation}

It follows from \eqref{eq:cusp.trunc.metric}, \eqref{eq:f.trunc.zero} and a change of coordinates $r = e^{\rho}$ that the metric $h_*$ extends smoothly as a flat metric across the cusp points $c_i$. The metric completion is
\begin{equation} \label{eq:cusp.trunc.smooth}
    \overline{(\YY, h_*)} =  (\SS^2, h_*) \text{ with } \bar h_* \text{ smooth},
\end{equation}
where by slight abuse of notation we identify $\YY = \SS^2 \setminus \{ c_1, c_2, c_3 \} \subset \SS^2$ using the identity map. See figure \ref{fig:starfish} below. 

\begin{figure}[h!]
    \centering
    \includegraphics[scale=0.8]{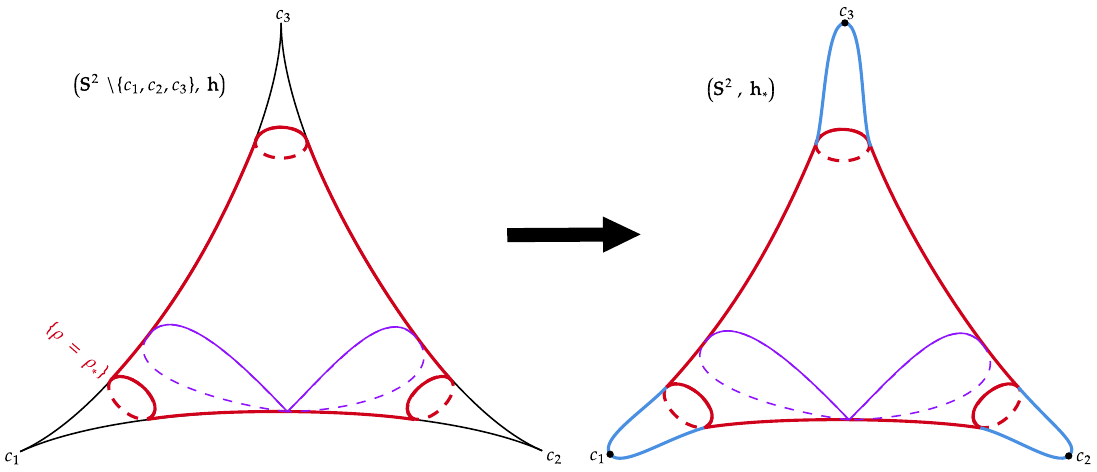}
    \caption{The construction of our smooth metric on $\SS^2$. Starting from the complete hyperbolic metric on $\SS^2 \setminus \{ c_i \}_{i=1}^3$, we modify the metric near the cusps (in coordinates, $\rho \leq \rho_*$) to form suitable caps. The tips of the caps correspond to $\{ c_i \}_{i=1}^3$, across which our metric extends smoothly. A figure-eight geodesic is highlighted.}
    \label{fig:starfish}
\end{figure}
\begin{theo} \label{theo:starfish.trunc.systole}
    For $\rho_* < \log 2 - \sys(\YY, h)$:
    
    A closed geodesic in $(\SS^2, h_*)$ has length $\sys(\SS^2, h_*)$ if and only if it is a geodesic, under the identity map, in the isometric region $(\YY \setminus \cup_{i=1}^3 \CC_i^{\leq \rho_*}, h)$.
\end{theo}
We rely on the following elementary geometric fact that follows from computing lengths in the hyperbolic metric \eqref{eq:cusp.metric}:

\textbf{Geometric fact}. If $\delta$ is a segment in $\CC_i \setminus \CC_i^{\leq \rho_*}$, connecting $\partial \CC_i$ to $\partial \CC_i^{\leq \rho_*}$, then
\[ \length(\delta, h) \geq \dist(\partial \CC_i, \partial \CC_i^{\leq \rho_*}, h) = \int_{\rho_*}^{\log 2} h(\partial_\rho, \partial_\rho) \, d\rho = \log 2 - \rho_*. \]

\begin{proof}
    First we show that:
    \begin{equation} \label{eq:starfish.trunc.systole.Y}
        \gamma \subseteq \YY \setminus \cup_{i=1}^3 \CC_i^{\leq \rho_*} \text{ for every systolic geodesic $\gamma$ of $(\YY, h)$.}
    \end{equation}
    We avoid using Theorem \ref{theo:starfish.systole} and opt for an elementary argument based on the Geometric Fact above that we re-use later.
    
    Suppose that \eqref{eq:starfish.trunc.systole.Y} failed and some systolic geodesic $\gamma$ in $(\YY, h)$ intersects, e.g., $\CC_1^{\leq \rho_*}$. Since
    \[ \length(\gamma, h) = \sys(\YY, h) < \log 2 - \rho_*, \]
    we deduce from the Geometric Fact above that $\gamma$ is fully contained in $\CC_1$. This contradicts the maximum principle since, by \eqref{eq:cusp.foliation},  $\CC_1$ is foliated by strictly $h$-convex curves. This completes the proof of \eqref{eq:starfish.trunc.systole.Y}. 

    Since all systolic geodesics of $(\YY, h)$ are contained in $\YY \setminus \cup_{i=1}^3 \CC_i^{\leq \rho_*}$ by \eqref{eq:starfish.trunc.systole.Y}, it follows from the isometry \eqref{eq:cusp.trunc.isometric} that they are candidates for $\sys(\SS^2, h_*)$. Thus:
    \begin{equation} \label{eq:starfish.trunc.systole.Y.leq.S2}
        \sys(\SS^2, h_*) \leq \sys(\YY, h).
    \end{equation}
    Now assume, for the sake of contradiction, that our proposition fails. There must exist a systolic geodesic $\gamma$ in $(\SS^2, h_*)$ which is not a systolic geodesic contained in the isometric region $(\YY \setminus \cup_{i=1}^3 \CC_i^{\leq \rho_*}, h)$. 
    
    We claim:
    \begin{equation} \label{eq:starfish.trunc.systole.contr}
        \gamma \cap (\cup_{i=1}^3 \CC_i^{\leq \rho_*}) \neq \emptyset.
    \end{equation}
    Otherwise,
    \[ \gamma \subset \SS^2 \setminus (\cup_{i=1}^3 \CC_i^{\leq \rho_*}) = (\YY \setminus \cup_{i=1}^3 \CC_i^{\leq \rho_*}) \cup (\cup_{i=1}^3 \{ c_i \}). \]
    Since $\cup_{i=1}^3 \{ c_i \}$ is a discrete subset of the right hand side and $\gamma$ is nonconstant, the inclusion would improve to
    \[ \gamma \subset \YY \setminus \cup_{i=1}^3 \CC_i^{\leq \rho_*}. \]
    Now using the isometry \eqref{eq:cusp.trunc.isometric} on $\YY \setminus \cup_{i=1}^3 \CC_i^{\leq \rho_*}$, and inequality \eqref{eq:starfish.trunc.systole.Y.leq.S2}, $\gamma$ would be a systolic geodesic in the isometric region $(\YY \setminus \cup_{i=1}^3 \CC_i^{\leq \rho_*}, h)$, a contradiction to our standing assumption. Therefore, \eqref{eq:starfish.trunc.systole.contr} must indeed hold.
    
    Without loss of generality, \eqref{eq:starfish.trunc.systole.contr} implies:
    \[ 
        \gamma \cap \CC_1^{\leq \rho_*} \neq \emptyset.
    \]
    It follows from the maximum principle and \eqref{eq:cusp.trunc.foliation} that $\gamma$ cannot be fully contained in $\CC_1 \cup \{ c_1 \} \subset \SS^2$. Therefore, $\gamma$ contains a geodesic segment $\delta \subset \CC_1 \setminus \CC_1^{\leq \rho_*}$ joining the curve $\partial \CC_1^{\leq \rho_*}$ to the curve $\partial (\YY \setminus \CC_1) = \partial \CC_1$. Using the isometry \eqref{eq:cusp.trunc.isometric} on $\CC_1 \setminus \CC_1^{\leq \rho_*}$, the Geometric Fact, and our choice of $\rho_*$, we deduce again:
    \[ \sys(\SS^2, h_*) = \length(\gamma, h_*) \geq \dist(\partial \CC_1, \partial \CC_1^{\leq \rho_*}, h) = \log 2 - \rho_* > \sys(\YY, h). \]
    This contradicts \eqref{eq:starfish.trunc.systole.Y.leq.S2}.
\end{proof}



\section{Theorems in introduction}

\begin{proof}[Proof of Theorem \ref{theo:intro.width}]
    Recall that $\sys(\SS^2, g)$ denotes the systolic length, and $\omega_{\textnormal{AP}}(\SS^2, g)$, $\omega_{\textnormal{GG}}(\SS^2, g)$ denote the Almgren--Pitts and Gromov--Guth widths. Clearly:
    \begin{equation} \label{eq:intro.width.GG.leq.AP}
        \omega_{\textnormal{GG}}(\SS^2, g) \leq \omega_{\textnormal{AP}}(\SS^2, g),
    \end{equation}
    since the former is an infimum over larger collection.

    It follows from the parallel theory of double-well phase transitions, specifically via Guaraco \cite{guaraco2018min} and Mantoulidis \cite{mantoulidis2021allen}, or Gaspar--Guaraco \cite{gaspar2018allen} and Chodosh--Mantoulidis \cite{chodosh2023p}, combined with Dey \cite{dey2022apac} to translate back to 1-cycles, that $\omega_{\textnormal{GG}}(\SS^2, g)$ is necessarily attained by a closed geodesic.  Therefore:
    \begin{equation} \label{eq:intro.width.sys.leq.AP}
        \sys(\SS^2, g) \leq \omega_{\textnormal{GG}}(\SS^2, g).
    \end{equation}

\begin{rema}
This inequality is not yet known with a direct argument. However, the analogous one for Almgren--Pitts widths is known by the work of Calabi--Cao \cite[Theorem 2.4]{calabi1992simple}. 
\end{rema}

    Combining \eqref{eq:intro.width.GG.leq.AP} and \eqref{eq:intro.width.sys.leq.AP} one gets the desired three-way inequality:
    \begin{equation} \label{eq:intro.width.leq.proof}
        \sys(\SS^2, g) \leq \omega_{\textnormal{GG}}(\SS^2, g) \leq \omega_{\textnormal{AP}}(\SS^2, g).
    \end{equation}
       
    For any {nonsimple} closed geodesic, the Calabi--Cao cut-and-paste construction \cite[Lemma 2.2]{calabi1992simple} 
    produces a 1-sweepout of $(\SS^2, g)$ with maximum length $\leq \sys(\SS^2, g)$. Having a nonsimple systolic geodesic to begin with implies:
    \begin{equation} \label{eq:intro.width.AP.leq.sys}
        \omega_{\textnormal{AP}}(\SS^2, g) \leq \sys(\SS^2, g).
    \end{equation}
    We conclude totally equality,
    \[ \omega_{\textnormal{AP}}(\SS^2, g) = \sys(\SS^2, g) = \omega_{\textnormal{GG}}(\SS^2, g) \] 
    by combining \eqref{eq:intro.width.leq.proof} with  \eqref{eq:intro.width.AP.leq.sys}.
\end{proof}


\begin{proof}[Proof of Theorem \ref{theo:intro.starfish} and Theorem \ref{theo:starfish.systole}]
    Invoke the construction of Theorem \ref{theo:starfish.trunc.systole} with any sufficiently negative $\rho_*$ and let $\gamma_*$ be any systolic geodesic of $(\SS^2, h_*)$. Theorem \ref{theo:starfish.trunc.systole} guarantees that $\gamma_*$ is an isometric image of a systolic geodesic $\gamma$ of our hyperbolic three-legged starfish $(\YY, h)$. 
    
    Using Proposition \ref{prop:hyperbolic}'s (1), we see that $\gamma$ is non-simple, and thus so is $\gamma_*$. Then, using Theorem \ref{theo:intro.width}, we see that $\gamma$ additionally attains the Almgren--Pitts (and Gromov--Guth) width. Therefore, \cite[Theorem 2.4]{calabi1992simple} guarantees that $\gamma$ has a single self-intersection. Finally, Proposition \ref{prop:hyperbolic}'s (2) guarantees it is (the isometric image of) one of the standard figure-eight geodesics in our hyperbolic three-legged starfish $(\YY, h)$. This completes the proof of Theorem \ref{theo:intro.starfish} with $g := h_*$. 
    
    The fact that $\gamma$ was, by construction, the isometric image of any systolic geodesic $(\YY, h)$ completes the proof of Theorem \ref{theo:starfish.systole}.
\end{proof}
\begin{rema} \label{rema:rho.infty}
    Sending $\rho_* \to -\infty$ in the contruction produces metrics $h_*$ that converge in the pointed $C^\infty_{\operatorname{loc}}$ sense to the original thrice punctured $\SS^2$ with its complete finite-area hyperbolic metric.
\end{rema}

\begin{rema}
We recall that the Calabi--Cao argument in \cite[Lemma 2.2]{calabi1992simple} constructs an explicit sweepout $\sigma:[0,1] \to \mathcal{Z}_1(\SS^2; \ZZ_2)$ that realizes a figure--eight geodesic $\gamma$ as its longest curve:
\[
\omega_{\textnormal{AP}}(\SS^2, h_*) = \omega_{\textnormal{GG}}(\SS^2, h_*) = \sys(\SS^2, h_*) = \max_{t \in [0,1]} \operatorname{length}(\sigma(t), h_*) = \operatorname{length}(\gamma, h);
\]
above, we compute the length of $\gamma$ with respect to $h$ because it is known to be in the isometrically hyperbolic region of $\SS^2$, and in particular inside $\YY$. Informally, the sweepout goes as follows. After relabeling, we may decompose $\gamma = \gamma_1 \gamma_2$ using geodesic loops $\gamma_1$, $\gamma_2$ with $[\gamma_1] = [c_1]$ and $[\gamma_2] = [c_2]^{-1}$ in $\pi(\YY)$, where $[c_1]$, $[c_2] \in \pi(\YY)$ indicate the free homotopy classes (consistently oriented) corresponding to the cusps $c_1$, $c_2$. Viewed in $\SS^2$, both $\gamma_1$, $\gamma_2$ bound domains $D_1$, $D_2 \subset \SS^2$ containing $c_1$, $c_2$ whose boundaries are geodesic loops with an acute angle. By this weak convexity of the boundary and the lack of shorter closed geodesics in $(\SS^2, h_*)$, one may sweep out $D_1$, $D_2$ (e.g., approximately with a curve shortening flow) with a continuous family $(\sigma(t))_{t \in [1/2, 1]}$ with $\sigma(1/2) = \gamma_1 \cup \gamma_2$ and $\sigma(1) = \{ c_1 \} \cup \{ c_2 \}$. On the flip side, $\gamma_1 \gamma_2^{-1}$ (note the parametrization reversal) also bounds a domain $D_{12}$ whose boundary consists of two geodesic segments and two acute angles. By the same argument, one may sweep out $D_{12}$ with a continuous family $(\sigma(t))_{t \in [0,1/2]}$ with $\sigma(1/2) = \gamma_1 \cup \gamma_2$ (we can use $+\gamma_2$ again rather than $-\gamma_2$ since $\sigma$ is a mod-2 1-cycle) and $\sigma(1) = \{ c_3 \}$. One may show that $\sigma$ induces a degree-1 map $\SS^2 \to \SS^2$, and thus a 1-sweepout.
%
\end{rema}

\bibliography{main}{}
\bibliographystyle{amsalpha}

\end{document}